\documentclass{amsart}

\usepackage{amssymb}

\newtheorem{theorem}{Theorem}
\newtheorem{pro}[theorem]{Proposition}
\newtheorem{lemma}[theorem]{Lemma}
\newtheorem{clm}[theorem]{Claim}
\newtheorem{cor}[theorem]{Corollary}

\theoremstyle{definition}

\theoremstyle{remark}
\newtheorem{remark}[theorem]{Remark}

\numberwithin{equation}{section}

\newcommand{\del}{\partial}

\def\opp{\ensuremath{\mbox{\rm Opp}}}
\def\bddM{\ensuremath{38t(M)}}

\begin{document}

\title[Strong cylindricality and the monodromy of bundles]{Strong cylindricality and\\ the monodromy of bundles} 

\author{Kazuhiro Ichihara}
\address{Department of Mathematics, College of Humanities and Sciences, Nihon University, 3-25-40 Sakurajosui, Setagaya-ku, Tokyo 156-8550, Japan}
\email{ichihara@math.chs.nihon-u.ac.jp}
\thanks{The first author is supported by JSPS KAKENHI Grant Number 23740061.}

\author{Tsuyoshi Kobayashi}
\address{Department of Mathematics, Nara Women's University, Kitauoya Nishimachi, Nara 630-8506, Japan}
\email{tsuyoshi@cc.nara-wu.ac.jp}
\thanks{The second author  is supported by JSPS KAKENHI Grant Number 25400091.}

\author{Yo'av Rieck}
\address{Department of mathematical Sciences, University of Arkansas, Fayetteville, AR 72701}
\email{yoav@uark.edu}



\date{\today}


\commby{Daniel Ruberman}

\begin{abstract}
A surface $F$ in a 3-manifold $M$ is called cylindrical if $M$ cut open along $F$ admits an essential annulus $A$. 
If, in addition, $(A, \partial A)$ is embedded in $(M, F)$, then we say that $F$ is strongly cylindrical.
Let $M$ be a connected 3-manifold that admits a triangulation using $t$ tetrahedra 
and $F$ a two-sided connected essential closed surface of genus $g(F)$.  
We show that if $g(F)$ is at least $38 t$, then $F$ is strongly cylindrical.  
As a corollary, we give an alternative proof of the assertion 
that every closed hyperbolic 3-manifold admits only finitely many fibrations over the circle with connected fiber 
whose translation distance is not one, which was originally proved by Saul Schleimer.
\end{abstract}

\maketitle

By {\it manifold} we mean connected 3-dimensional manifold; 
any manifold considered is assumed to be {\it tame}, 
that is, obtained from a compact manifold by removing finitely many compact connected sets from its boundary.
All surfaces considered are assumed to be closed and connected; 
every embedded surface is assumed to be two sided.  
We assume familiarity with the basic notions of 3-manifold topology
and in particular the basic notions of normal surface theory and fibered manifolds, 
where by {\it fibered manifold} we mean a manifold that fibers over $S^{1}$.  
A connected surface properly embedded in a manifold is called {\it essential} 
if it is incompressible, boundary incompressible, not boundary parallel, and not a ball-bounding sphere. 
A surface $F$ in a manifold $M$ is called {\it cylindrical} if $M$ cut open along $F$ admits an essential annulus $A$.
If, in addition, $(A,\del A)$ is embedded in $(M,F)$, then $F$ is called {\it strongly cylindrical}
(thus $F$ is strongly cylindrical if, after gluing $M$ cut open along $F$ to obtain $M$, the components of $\del A$ are disjoint).
Note that if $M$ fibers with fiber $F \not\cong S^{2}$ then $F$ is cylindrical but not necessarily strongly cylindrical.  
We use $g(F)$ to denote the genus of a surface $F$ and 
$t(M)$ to denote the minimal number of tetrahedra required to triangulate a 3-manifold $M$; 
we allow ideal and truncated vertices as described in~\cite{KobayashiRieck}.  
It is well known that every tame manifold admits such a triangulation.

Our goal is to provide a simple, self contained proof of the following:

\begin{theorem}
\label{thm:easy}
Let $M$ be a connected tame 3-manifold that admits a triangulation using $t(M)$ tetrahedra 
and $F \subset M$ a two sided connected essential closed surface.
If $g(F) \geq \bddM$, then $F$ is strongly cylindrical.
\end{theorem}

For the next two corollaries 
suppose that $M$ fibers over $S^1$ with fiber $F$.  
Suppose in addition that $F$ is strongly cylindrical, and
let $(A,\del A)$ be an annulus embedded in $(M,F)$ as in the definition of strong cylindricality.  
We may identify $M$ cut open along $F$ with $F \times I$, where $I =[0,1]$.
It is easy to see $A \subset F \times I$ is isotopic to $\alpha \times I$,
where $\alpha$ is an essential curve of $F$.  Thus we see that  
$\alpha \times \{0\}$ is disjoint from $\alpha \times \{1\}$ 
in $M$ (perhaps after isotopy).
Since $\alpha \times \{1\}$ is the image of  $\alpha \times \{0\}$ under the monodromy
Theorem~\ref{thm:easy} implies:

\begin{cor}
\label{cor:fiber}
Let $M$ be a closed connected 3-manifold that fibers over $S^1$ with fiber $F$.
If $g(F) \geq \bddM$, then there is an essential simple closed curve
on $F$ that is disjoint from its image under the monodromy (perhaps after isotopy).
\end{cor}

For the next corollary we assume that the reader is familiar with the curve complex of $F$,
denoted by $\mathcal{C}(F)$, where here $F$ is the fiber of $M$ 
as above (for definitions, see, for example, \cite{BS}).
The monodromy induces an isometry on $\mathcal{C}(F)$, and the conclusion 
of Corollary~\ref{cor:fiber} is equivalent to saying that the translation distance of
this isometry is at most one.    
If we assume that $M$ is hyperbolic, then by~\cite{MR744850} 
$M$ admits only finitely many isotopy classes of essential surfaces
(and in particular fibers) of genus less than  $\bddM$.  
This together with Corollary 2 gives the following corollary, also proved by Schleimer~\cite{schleimer}:  

\begin{cor}
\label{cor:hyper}
A closed hyperbolic manifold admits only finitely many fibrations over $S^{1}$
with connected fiber whose translation distance is not one. 
\end{cor}

Every complete finite volume hyperbolic 3-manifold is finitely covered by a
manifold that fibers over $S^{1}$ and has first Betti
number at least 2 (see  Agol~\cite{agol} for the closed case
and Wise~\cite{WiseHierarchy} for the non compact case;
this uses the work of many people,
see~\cite{wise} and references therein).
Any such cover admits infinitely many fibrations with connected fiber;
thus Corollary~\ref{cor:hyper} is far from vacuous.

\begin{remark}  A few results from the literature should be mentioned:
\begin{enumerate}
\item Hass~\cite{hass} shows that 
for an essential surface $F$ in a closed hyperbolic manifold $M$, if the 
genus of $F$ is large and the volume of $M$ is small, then $F$ is cylindrical. 
This is closely related to Theorem~\ref{thm:easy}, as
it is known that there exists a constant $k$ so that 
$M$ is obtained by filling a manifold $X$ for which
$t(X) \leq k \mbox{vol}(M)$ holds (see, for example,~\cite{JT}).
\item In~\cite{sela} Sela generalizes~\cite{hass} and obtains a group theoretic result.
This impressive and far reaching work is beyond the scope of our discussion.
\item Schleimer~\cite{schleimer} obtains a result similar to Theorem~\ref{thm:easy}; 
the contributions of the current paper are the explicit linear bound and the simple proof
(Schleimer's proof relies on deep results about normal surfaces). 
\item Eudave--Mu\~noz and Neumann--Coto~\cite{EC} have a simple and elementary
argument showing that if $F \subset M$ is
essential and $g(F) \geq t(M)+1$, then
$F$ is cylindrical.  However it does not appear that their techniques can be used
to prove strong cylindricality, necessary for Corollaries~\ref{cor:fiber} and~\ref{cor:hyper}.
\item Bachman and Schleimer~\cite{BS} studied the translation distance of the isometry
induced by the monodromy on $\mathcal{C}(F)$ showing, among other things,
that for any manifold $M$
there is $d>0$ so that for any fibration of $M$ the translation distance is at
most $d$.  For hyperbolic manifolds, this follows from Corollary~\ref{cor:hyper}.
\end{enumerate}
\end{remark}

\section{Setting up the proof of Theorem~\ref{thm:easy}}
\label{sec:EasySetUp}

Let $F \subset M$ be a two sided connected essential closed surface. 
Let $\mathcal{T}$ be a triangulation of $M$ using $t = t(M)$ tetrahedra.
The set up is based on~\cite{KobayashiRieck} and we refer the reader to that paper for more
details.  

Isotope $F$ to be normal with respect to $\mathcal{T}$.  Let $N$ be a closed regular neighborhood of
the 1-skeleton $\mathcal{T}^{(1)}$.  Then $F \cap N$ consists of disks
that we call {\it vertex disks}. 
Given a normal disk $D$ in a tetrahedron $T$, we call $\mbox{cl}(D \setminus N)$ 
a {\it truncated normal disk}; when $D$ is 
a normal triangle (respectively quadrilateral),   
we call $\mbox{cl}(D \setminus N)$ a {\it truncated normal triangle}
(respectively {\it quadrilateral}).  Note that a truncated normal triangle
is a hexagon and a truncated normal quadrilateral is an octagon.
We refer to truncated normal disks and vertex disks as {\it faces}.
The union of the faces forms a cell decomposition of $F$, and the 
1-skeleton of this decomposition is a trivalent graph, say $G$.

We decompose the faces into parallel families as follows
(in~\cite{KobayashiRieck}, parallel families are called {\it $I$-equivalent families}).
The intersection of $F$ with each tetrahedron of $\mathcal{T}$ consists of five families of parallel
normal disks: four families of normal triangles and one family of normal quadrilaterals
(each family may be empty).  The truncated normal disks that correspond to 
each parallel family form a parallel family of faces.
The vertex disks along each edge form a parallel family of faces.  

We color faces in each parallel family as follows
(this is a simplified version of the coloring used in~\cite{KobayashiRieck}).  
The outermost faces are colored red.  The remaining
faces are colored alternately yellow and blue (this coloring is not unique).

Since every red vertex disk is outermost along an edge of $\mathcal{T}^{(1)}$,
all the faces around it are red  truncated normal disks.  We record this:

\begin{lemma}
\label{lem:RedVertexDisks}
All the faces around a red vertex disk are red truncated normal disks.
\end{lemma}

Let $R$, $B$ and $Y$ denote the union of the red, blue and yellow faces, respectively.
Since $G$ is trivalent, $R$, $B$ and $Y$ are 
subsurfaces of $F$.   We use $| \ \ |$ to denote the number of components.

\begin{lemma}
\label{lem:R}
$\chi(R) \geq -(22t-1)$ and $|\del R| \leq 22t-1$.
\end{lemma}

\begin{proof}
Let $F_{1,},\dots,F_{k},F_{k+1},\dots,F_{n}$ be the faces of $R$ (for some $k,n$) ordered so that the red vertex disks 
are $F_{k+1},\dots F_{n}$.  Since the decomposition of $R$ is along a trivalent graph (namely $G \cap R$), 
for each $i$ ($1 \le i \le n$),
$F_{i} \cap (\cup_{j < i} F_{j})$ is either empty, or consists of
intervals, or is all of $\del F_{i}$.  The number of intervals is at most $3$ when
$F_{i}$ is a truncated normal triangle and at most $4$ when
$F_{i}$ is a truncated normal quadrilateral.  
It follows easily that when $F_{i}$ is a truncated normal triangle we have:
$$\chi(\cup_{j \leq i} F_{j}) \geq \chi(\cup_{j < i} F_{j}) - 2 \mbox{ and } |\del (\cup_{j \leq i} F_{j})| \leq |\del (\cup_{j < i} F_{j})| + 2.$$ 
When $F_{i}$ is a truncated normal quadrilateral we have:
$$\chi(\cup_{j \leq i} F_{j}) \geq \chi(\cup_{j < i} F_{j}) - 3 \mbox{  and  } |\del (\cup_{j \leq i} F_{j})| \leq |\del (\cup_{j < i} F_{j})| + 3.$$ 
By Lemma~\ref{lem:RedVertexDisks}, for $i \geq k+1$, $F_{i}$ caps off a boundary component of  $\cup_{j < i} F_{j}$.
Thus:
$$\chi(\cup_{j \leq i} F_{j}) = \chi(\cup_{j < i} F_{j}) +1 \mbox{ and } |\del (\cup_{j \leq i} F_{j})| = |\del (\cup_{j < i} F_{j})| - 1.$$ 
Since every tetrahedron contains 
at most four families of truncated triangles and at most one family of truncated quadrilaterals
and each family has at most two red faces, 
there are at most $8t$ red truncated normal triangles and at most 
$2t$ red truncated quadrilaterals.  
Combining these facts we see that:
$$\chi(R) = \chi(\cup_{j \leq n} F_{j})  > \chi(\cup_{j \leq k} F_{j}) \geq (-2)8t + (-3)2t = -22t$$
and 
$$|\del R| = |\del (\cup_{j \leq n} F_{j})| < |\del (\cup_{j \leq k} F_{j})| \leq (2)8t + (3)2t = 22t.$$
In both cases the second inequality is strict since there is at least one red vertex disk.
The lemma follows.
\end{proof}

Below we construct $F_{0},\dots,F_{4}$.  The reader can easily verify that all are subsurfaces of $F$.
We start with $F_{0}$ which is defined to be $Y \cup B$.
Since $G$ is trivalent $F_{0}$ is a subsurface of $F$.

\begin{lemma}
\label{lem:F0}
$\chi(F_{0}) \leq  -(54t-1)$ and $|\del F_{0}| \leq 22t-1$.
\end{lemma}

\begin{proof}
Note that $F = R \cup F_{0}$ and $F_{0} \cap R = \del F_{0} = \del R$ consists of simple closed curves.
Hence $\chi(F) = \chi(F_{0}) + \chi(R)$.  By assumption, $g(F) \geq \bddM$; equivalently, $\chi(F) \leq -76t+2$.  
By Lemma~\ref{lem:R}, we 
have that $\chi(F_{0}) \leq -(54t-1)$.

Since $\del F_{0} = \del R$, by Lemma~\ref{lem:R} we have that $|\del F_{0}| \leq 22t-1$.
\end{proof}

Let $\Gamma_{1} \subset F_{0}$ denote the arc components of $B \cap Y$
(that is, the components that are not simple closed curves).  
The endpoints of $\Gamma_{1}$ are exactly the vertices of $G$ where faces
of three distinct colors meet.  Let $\mathcal{V}$
denote the vertices of the red truncated normal disks.  
By Lemma~\ref{lem:RedVertexDisks} the endpoints of $\Gamma_{1}$
are contained in $\mathcal{V}$.  
We denote $\mathcal{V}_{+} \subset \mathcal{V}$ the vertices where
three colors meet, and $\mathcal{V}_{-} \subset \mathcal{V}$ the vertices where one
face is red and the other two are both blue or both yellow.
Then the endpoints of $\Gamma_{1}$ are exactly $\mathcal{V}_{+}$.
By exchanging the yellow and blue coloring along the vertex disks
we obtain a new coloring where the roles of $\mathcal{V}_{+}$
and $\mathcal{V}_{-}$ are exchanged (note that at every vertex of $G$
exactly one face is a vertex disk); 
hence we may assume
that $|\mathcal{V}_{+}| \leq |\mathcal{V}_{-}|$; this implies that
$|\mathcal{V}_{+}| \leq \frac{1}{2}|\mathcal{V}|$.
Since every arc of $\Gamma_{1}$ has two endpoints we get: 
$$|\Gamma_{1}| = \frac{1}{2} |\mathcal{V}_{+}| \leq \frac{1}{4} |\mathcal{V}| \leq \frac{1}{4}(6 \cdot 8t + 8 \cdot 2t) = 16t.$$
(Here we used that there are at most $8t$ red
truncated triangles and each is a hexagon, 
and at most $2t$ red truncated quadrilaterals and each is an octagon.)

Let $F_{1}$ be the surface obtained by cutting $F_{0}$ open along $\Gamma_{1}$.  Then
we have:

\begin{lemma}
\label{lem:F1}
$\chi(F_{1}) \leq -(38t-1)$ and $|\del F_{1}| \leq 38t-1$.
\end{lemma}

\begin{proof}
Cutting along an arc increases the Euler characteristic by one and changes the 
number of boundary components by $\pm 1$.  Above we saw that
$|\Gamma_{1}| \leq 16t$; this and Lemma~\ref{lem:F0} imply the lemma.
\end{proof}

By Lemma~\ref{lem:F1} we have that $\chi(F_{1}) \leq -|\del F_{1}|$.
Since the Euler characteristic and the number of boundary components are
both additive under disjoint union, there is a component of $F_{1}$, say $F_{2}$,
so that  $\chi(F_{2}) \leq -|\del F_{2}|$.
Let $\Delta_{1}$ be the components of $F \setminus F_{2}$ that are disks
 (possibly, $\Delta_1 = \emptyset$).  Let
$F_{3}$ be $F_{2} \cup \Delta_{1}$.  Since attaching a disk increases the Euler characteristic 
by one and decreases the number of boundary components by one, 
we have that $\chi(F_{3}) \leq -|\del F_{3}|$.

\begin{lemma}
\label{lem:F3}
$\chi(F_{3}) < 0$ and $\del F_{3}$ is essential in $F$.
\end{lemma}

\begin{proof}
We first show that $\chi(F_{3}) < 0$.  
Since $\chi(F_{3}) \leq -|\del F_{3}|$, we may assume that $|\del F_{3}|  = 0$.
Then $F_{3}$ is closed and hence 
$F_{3} = F$; thus $\chi(F_{3}) = \chi(F) \leq -76t + 2 \leq -74<0$.

We easily see that $\del F_{3}$ is essential in $F$:
for a contradiction, suppose there is a component
$\gamma$ of $\del F_{3}$ that is inessential in $F$.  Then $\gamma$ bounds a disk, say $D$, in $F$.
By construction, $\gamma$ is a component of $\del F_{2}$.  Since $\chi(F_{2}) \leq -|\del F_{2}|$,
we have that $g(F_{2}) \geq 1$.  Hence $F_{2} \not\subset D$; thus $D$ is a component of $\Delta_{1}$.
Since $\Delta_{1} \subset F_{3}$ and $F_{2} \subset F_{3}$, 
we see that $\gamma \not \subset \del F_{3}$, contradiction.
\end{proof}

Let $\mathcal{E}$ be the components of  $(B \cap Y) \cap F_{3}$ that are essential 
simple closed curves in $F$.  
Since the Euler characteristic is additive when gluing along
simple closed curves, for some component of $F_{3}$ cut open along $\mathcal{E}$, say $F_{4}$, 
we have that $\chi(F_{4}) < 0$. 
  
For the convenience of the reader we summarize our construction so far;
we use ``c.o.a.'' for ``cut open along'' and ``comp'' for ``a component of'':

\begin{equation}
\label{construction}
B \cup Y = F_{0}  \stackrel{\mbox{\tiny c.o.a. }\Gamma_{1}}\longrightarrow F_1 \stackrel{\mbox{\tiny comp}}\longrightarrow F_{2} \stackrel{\cup \Delta_1}\longrightarrow F_{3}
\stackrel{\mbox{\tiny c.o.a. }\mathcal{E},  \mbox{\tiny comp}}\longrightarrow F_4
\end{equation}

We say that $F_{4}$ is {\it essentially yellow} (respectively {\it essentially blue}) if
$F_{4}$ colored yellow (respectively blue) except, perhaps, for a subset that is contained 
in a disk in the interior of $F_{4}$.

\begin{lemma}
\label{lem:F4}
The following conditions hold:
	\begin{enumerate}
	\item $\chi(F_{4}) < 0$.
	\item $i_{*}:\pi_{1}(F_{4}) \to \pi_{1}(F)$ is injective, where here
	$i_{*}$ is the homomorphism induced by the inclusion.
	\item $F_{4}$ is essentially yellow or essentially blue.
	\end{enumerate}
\end{lemma}

\begin{proof}
Condition~(1) was established before the lemma.

By construction, every component of $\del F_{4}$ is either a component of $\del F_{3}$
or a component of $\mathcal{E}$.  It follows that every component of $\del F_{4}$ is essential in $F$;~(2)
follows.

Let $\Gamma_{2}$ be the components of $(B \cap Y) \cap \mbox{int}(F_{4} \setminus \Delta_{1})$
(for this argument,~(\ref{construction}) above is helpful).
We claim that every component of $\Gamma_{2}$, say $\gamma_{2}$, is a simple closed curve
that bounds a disk in $F_{4}$.
By construction $\gamma_{2}$ is in $F_{3}$, and in fact in
$F_{2} = F_{3} \setminus \mbox{int}(\Delta_{1})$; since $F_{2}$ is a component of $F_{1}$,
we have that $\gamma_{2} \subset F_{1}$.
Since $F_{1}$ was obtained by cutting $F_{0}$ open along $\Gamma_{1}$, we see that
$\gamma_{2}$ is a simple closed curve.  
Since $\gamma_{2} \not\in \mathcal{E}$, it is inessential in $F$.
Hence $\gamma_{2}$ bounds a disk, say $D_{2}$, in $F$.  
Lemma~\ref{lem:F4}~(2) shows that $D_{2} \subset F_{4}$.

Let $\Delta_{2} \subset F_{4}$ be the disks bounded by outermost curves of $\Gamma_{2}$.
Clearly, $\Delta_{2} \subset \mbox{int}(F_{4})$ consists of disjointly embedded disks. 
Let $D_{1} \in \Delta_{1}$ and $D_{2} \in \Delta_{2}$ be disks.  
Since $\del D_{2} \subset \Gamma_{2}$, it is disjoint from $D_{1}$; 
either $D_{1} \subset D_{2}$
or $D_{1} \cap D_{2} = \emptyset$.  
It follows that $\Delta_1 \cup \Delta_{2}$ consists of disks disjointly embedded in $\mbox{int}(F_{4})$;
thus $\Delta_{1} \cup \Delta_{2}$ is contained in a disk in $\mbox{int}(F_{4})$.  

By construction, $F_{4} \setminus\mbox{int}(\Delta_{1} \cup \Delta_{2})$ 
is disjoint from $B \cap Y$ and has no red points; hence it
is entirely yellow or entirely blue, completing the proof of the lemma.
\end{proof}

By~(1) of Lemma~\ref{lem:F4}, there exists a pair of pants $X \subset F_{4}$
so that $\del X$ is essential in $F_{4}$.  By~(2), $\del X$ is essential in $F$.
By~(3), after isotopy of $X$ in $F_{4}$ if necessary we may assume that
$X \subset \mbox{\rm int}Y$ or $X \subset \mbox{\rm int}B$.  Thus we proved:

\begin{pro}
\label{pro:X}
There exists a pair of pants $X \subset \mbox{\rm int}(Y)$ 
or $X \subset \mbox{\rm int}(B)$ so that 
$\del X$ is essential in $F$.
\end{pro}

\section{Proof of Theorem~\ref{thm:easy}}
\label{sec:EasyProof}

We fix the notation of the previous section.
Recall that by assumption $F$ is two-sided; we fix a coorientation on $F$.
By Proposition~\ref{pro:X} there exists a pair of pants 
$X \subset \mbox{int}(Y)$ or
$X \subset \mbox{int}(B)$ (say the former) so that $\del X$ is essential in $F$.
Hence each point of 
$X$ co-bounds $I$-bundles (the other portion of the boundaries lie elsewhere) on both sides 
for each $p \in X$,
let $I_{p}$ be the $I$-fiber that $p$ bounds into the positive side of $X$, that is,
$I_{p}$ is an embedding of $[0,1]$ given by the parallelism in a tetrahedron
containing $p$, so that $\{0\}$ is identified with $p$ and $I_{p}$
points into the positive direction as given by the coorientation of $F$.
The point corresponding to $\{1\}$ is called the point {\it opposite} $p$.
Sending a point $p \in X$ to its opposite point induces a map, say $\opp:X \to F$.  

\begin{clm}
\opp\ is one-to-one.
\end{clm}

\begin{proof}
For a contradiction, assume there exist distinct points $p_{1},p_{2} \in X$ 
so that $\opp(p_{1}) = \opp(p_{2})$.
By construction, $p_{1}$ and $p_{2}$ are contained in parallel faces of $F$,
and moreover, $\opp(p_{1})$ is contained in the unique face that separates the face containing $p_{1}$ 
from that containing $p_{2}$ within the parallel family.  
We see that $I_{p_{1}}$ and $I_{p_{2}}$ point to the face containing
$\opp(p_{1})$ from opposite directions.

The intersection of $X$ with the faces of $F$ induces cell decompositions on $X$.
Note that for each cell $S$ of $X$, we have that $\opp(S)$ is a cell of $\opp(X)$.
For $q \in S$, $I_{q}$ induces a coorientation on $\opp(S)$
(clearly, this is independent of $q \in S$); 
we may assume for convenience that it coincides with the coorientation
induced  by $F$ (the other case is similar).
By construction, if $S'$ is a face of $X$ that intersects $S$ in an edge, then the
coorientation induced on $\opp(S')$ by $I_{q'}$ (for $q' \in S'$)
agrees with that induced on $\opp(S')$ 
by $F$.  Connectivity of $X \setminus V$
(where here $V$ denotes the vertices on $X$) implies that 
for all $p \in X$, the coorientation
induced by $I_{p}$ on the face of $\opp(X)$ containing $\opp(p)$
agrees with that induced by $F$.

Now let $S_{p_{1}}$ and $S_{p_{2}}$ be the faces of $X$ containing the points $p_1,p_2$ above.
By the discussion above, $\opp(S_{p_{1}}) = \opp(S_{p_{2}})$, and the coorientation
induced on $\opp(S_{p_{1}})$ by $I_{p_{1}}$ is opposite that induced by $I_{p_{2}}$.
This contradiction completes the proof of the claim.
\end{proof}

We see that $\opp(X)$ is a pair of pants.  Since $X \subset \mbox{int}(B)$,
every point of $X$ is colored only blue.  Thus every point of $\opp(X)$ is colored 
yellow or red (possibly both).  We conclude that $X \cap \opp(X) = \emptyset$.
It now follows that 
$$\cup_{p \in X} I_{p}$$
is a trivial $I$-bundle over $X$ embedded in $M$, which we will denote by $X \times I$.

We consider the image of $X \times I$ in $M$ cut open along $F$, still denoted $X \times I$.
Denote the components of $\del X$ as $\alpha_{1}$, $\alpha_{2}$, and $\alpha_{3}$. 
Since $X \times I$ is a trivial $I$-bundle, the induced
$I$-bundle over $\alpha_{i}$ is an annulus, say $A_{i}$.  If $A_{i}$ is 
boundary parallel in $M$ cut open along $F$, then
the annulus it cobounds with the boundary does not contain $X$ (since $\del X$ is essential in $F$); 
it is easy to see
that if all three annuli $A_{1}$, $A_{2}$, and $A_{3}$ are boundary parallel then $F$ 
is the union of $X$, $\opp(X)$,
and three annuli; thus $g(F) = 2$, contradicting our assumption
that $g(F) \geq \bddM \geq 38$.  We conclude that at least one of the annuli, say $A_{1}$, 
is not boundary parallel.  Since one boundary component of $A_1$, namely
$\alpha_1$, is essential in $F$, $A_1$ is incompressible.  If $A_1$ boundary
compresses then it is obtained by banding a boundary parallel disk to
itself; it is easy to see that $A_1$ is boundary parallel or compressible in that
case.  We conclude that $A_1$ is essential, completing the proof of Theorem~\ref{thm:easy}.

\section*{Acknowledgment}
The authors would like to thank the anonimous referee for his/her valuable suggestions. 
They also thank Michael Harris for his careful reading of earlier drafts.

\bibliographystyle{amsplain}

\bibliography{biblio} 

\end{document}